\documentclass[11pt,letterpaper]{amsart}

\usepackage{amsmath, amssymb, amsthm}
\usepackage{enumitem}
\usepackage{hyperref}
\usepackage{xcolor}
\usepackage{comment}
\usepackage{todonotes}
\theoremstyle{plain}
\newtheorem{theorem}{Theorem}[section]
\newtheorem{lemma}[theorem]{Lemma}
\newtheorem{proposition}[theorem]{Proposition}
\newtheorem{corollary}[theorem]{Corollary}

\theoremstyle{definition}
\newtheorem{definition}[theorem]{Definition}

\newtheorem{remark}[theorem]{Remark}

\newcommand{\C}{\mathbb{C}}
\newcommand{\CP}{\mathbb{CP}}
\newcommand{\R}{\mathbb{R}}

\newcommand{\mG}{\mathcal{G}}
\newcommand{\mH}{\mathcal{H}}

\newcommand{\mL}{\mathcal{L}}
\newcommand{\mR}{\mathcal{R}}

\newcommand{\inn}{\langle \cdot , \cdot \rangle}

\DeclareMathOperator{\Id}{Id}
\DeclareMathOperator{\tr}{tr}
\DeclareMathOperator{\End}{End}

\DeclareMathOperator{\codim}{codim}

\title[Hirzebruch form and logarithmic connections]{The Hirzebruch quadratic form of a hyperplane arrangement and flat logarithmic connections}

\date{}

\author{Martin de Borbon}
\address{Department of Mathematical Sciences, Loughborough University, Schofield Building, Loughborough LE11 3TU, United Kingdom}
\email{m.de-borbon@lboro.ac.uk}

\author{Dmitri Panov}
\address{Department of Mathematics, King's College London, Strand, London WC2R 2LS, United Kingdom}
\email{dmitri.panov@kcl.ac.uk}


\begin{document}

\maketitle

\begin{abstract}
We prove that the Hirzebruch quadratic form of a complex hyperplane arrangement is non-positive on the set of stable weights, and we identify the zero locus within this set with flat logarithmic connections of a distinguished type. The proof uses Kempf--Ness and the frame-potential inequality.
\end{abstract}

\section{Introduction}

A complex hyperplane arrangement carries a natural quadratic form on its space of weights. For line arrangements in \(\CP^2\), this form appears in
Hirzebruch's work on the Miyaoka--Yau inequality and on the construction of ball quotients \cite{hirzebruch}; in higher dimensions it was introduced in \cite{borbonpanov}. This paper places the Hirzebruch quadratic form in the study of flat logarithmic connections.

Hirzebruch's construction of ball quotients from line arrangements uses the Miyaoka--Yau inequality together with the Aubin--Yau existence theorem for K\"ahler--Einstein metrics. The complex-hyperbolic structure of a ball quotient is encoded by a flat connection. The result below gives an arrangement-level counterpart: in
the equality case, one obtains a flat logarithmic connection directly from the weighted arrangement, by a finite-dimensional moment-map argument.

This gives a novel way of uncovering flat logarithmic connections from stability. The familiar Kobayashi--Hitchin correspondence starts from stable
vector bundles, or stable Higgs bundles, and produces Hermitian--Einstein or harmonic metrics which, under appropriate topological vanishing assumptions, give flat connections. In its differential-geometric form, this involves solving a PDE. In the present paper, the stable object is instead a weighted configuration of points in projective space. The vanishing of the Hirzebruch quadratic form is the additional
constraint which turns the moment-map solution into a flat logarithmic connection, while bypassing PDE methods.

\subsection{Main result}
Let \(\mathcal H=\{H_1,\ldots,H_n\}\) be a
hyperplane arrangement in \(\CP^d\), with \(d\geq 1\), and let
\(Q:\R^n\to\R\) be its Hirzebruch quadratic form, see
Definition~\ref{def:hqf}. For a positive weight vector
\(a=(a_1,\ldots,a_n)\in\R^n_{>0}\), we write \((\mathcal H,a)\) for the corresponding weighted arrangement. 
%Our main result is the following.

\medskip

\noindent\textbf{Theorem~A.}
\emph{Suppose that \((\mathcal H,a)\) is stable. Then}
\begin{equation}\label{eq:hirineq}
Q(a)\leq 0 \,.    
\end{equation}
\emph{Moreover, equality holds if and only if there exists a Hermitian metric \(\inn\) on \(\C^{d+1}\), unique up to scale, such that the logarithmic connection on \(T\C^{d+1}\), written in the standard trivialization,}
\begin{equation}\label{eq:connection}
\nabla=d-\sum_{i=1}^n A_i\frac{d\ell_i}{\ell_i}
\end{equation}
\emph{is flat. Here \(\ell_i\) is a defining linear equation for \(H_i\), and \(A_i=a_iP_i\), where \(P_i\) denotes the orthogonal projection onto \(H_i^\perp\) with respect to \(\inn\).}

\medskip

The inequality \eqref{eq:hirineq} is one of the main results of \cite{borbonpanov}, extending the case \(d=2\) previously proved in \cite{panov},  both using the theory of parabolic bundles \cite{mochizuki}.
Here we give a shorter, self-contained proof; the main new contribution is an equality characterization in terms of flat logarithmic connections determined by a distinguished Hermitian metric, unique up to scale.
In the terminology of Couwenberg--Heckman--Looijenga \cite{chl}, the connections appearing in Theorem~A are called \emph{Dunkl connections}.
Their work develops this class of logarithmic connections and gives important geometric examples, especially those associated with complex reflection arrangements. By contrast, Theorem~A yields a combinatorial criterion for the existence of such connections for an arbitrary hyperplane arrangement.

The proof of Theorem~A is based on the Kempf--Ness theorem for weighted configurations of points in \(\CP^d\). Stability gives a balanced Hermitian metric, or equivalently a Hermitian metric such that the weighted normal vectors to the hyperplanes form a tight frame. The theorem then follows from the frame-potential inequality \cite[Theorem 6.1]{waldron} applied to the codimension-two localizations of \(\mH\).
As a result, the Hermitian metric in Theorem~A has a moment-map interpretation: it
is the zero of a moment map for the diagonal action of \(SL(d+1,\C)\) on \((\mathbb{CP}^{d})^n\).  
This places the proof in the same moment-map framework as other work on point configurations in projective space, such as Hughston--Salamon \cite{hughstonsalamon}.

\begin{remark}
In dimension \(d=1\), the Hirzebruch quadratic form is identically zero, so Theorem~A reduces to the existence of a Hermitian metric for which the weighted orthogonal projections \(a_iP_i\) add up to a scalar multiple of the identity. 
This case was proved in \cite[Proposition 2.3]{dunklC2}: such a metric exists precisely when the weighted configuration is stable. \footnote{The main result of \cite{dunklC2} shows that, in the generic case, Dunkl connections on \(T\C^2\) do not preserve non-zero Hermitian forms.}
\end{remark}

\subsection{Notation}
Throughout this paper, \(\mathcal H=\{H_1,\ldots,H_n\}\) denotes a hyperplane arrangement in \(\CP^d\), i.e. a finite set of pairwise distinct complex hyperplanes. We use the same symbol \(H_i\) for the corresponding linear hyperplane in \(\C^{d+1}\). When weights are present, we write \(a=(a_1,\ldots,a_n)\in\R^n_{>0}\) and denote by \((\mathcal H,a)\) the corresponding weighted arrangement.

\subsection{Outline}
In Section~\ref{sec:background} we recall the necessary background. In Section~\ref{sec:proof} we prove Theorem~A. In Section~\ref{sec:applications} we give three applications: a higher-dimensional analogue of Langer's bound on multiplicities of line arrangements \cite[Proposition 11.3.1]{langer}, a matroid-polytope criterion for the existence of flat logarithmic connections, and a localization result used in our construction of polyhedral K\"ahler metrics \cite{borbonpanov-pk}.

\section{Background}\label{sec:background}

\subsection{Hyperplane arrangements}\label{sec:ha}

Let \(\mathcal L\) be the intersection poset of \(\mH\). The elements of \(\mathcal L\) are the projective linear subspaces \(L\subset \CP^d\) obtained by intersecting members of \(\mathcal H\). The rank of \(L\in\mathcal L\) is
\[
r(L):=\codim L.
\]
The arrangement \(\mathcal H\) is called essential if \(\emptyset\in\mathcal L\), and we take the convention that \(\codim\emptyset=d+1\). The localization of \(\mathcal H\) at \(L\in\mathcal L\) is
\[
\mathcal H_L=\{H\in\mathcal H\mid L\subset H\}.
\]
The multiplicity of \(L\) is
\[
m_L:=|\mathcal H_L|.
\]

The arrangement \(\mathcal H\) is called reducible if there is a pair of proper subspaces \(P,Q\in\mathcal L\) such that \(\CP^d=P+Q\) and every \(H_i\in\mathcal H\) contains either \(P\) or \(Q\). Otherwise \(\mathcal H\) is irreducible. Similarly, \(L\in\mathcal L\) is called reducible or irreducible according to whether the localized arrangement \(\mathcal H_L\) is reducible or irreducible.

We denote by \(\mathcal G\) and \(\mathcal R\) the subsets of irreducible and reducible elements of \(\mathcal L\), respectively. We write \(\mathcal L^k\), \(\mathcal G^k\), and \(\mathcal R^k\) for the subsets of elements of rank \(k\). In rank two there is a simple criterion:
\[
L\in\mathcal G^2 \Longleftrightarrow m_L\geq 3\,,
\qquad
L\in\mathcal R^2 \Longleftrightarrow m_L=2 \,.
\]

\subsection{Stability and the Kempf-Ness Theorem} \label{sec:stab}

We first recall the stability condition for weighted arrangements.
\begin{definition}
We say that \((\mH, a)\) is \emph{stable} if for every non-empty proper element \(L \in \mL\) the following holds:
\begin{equation}\label{eq:stable}
	\frac{1}{r(L)}  \sum_{i \,|\, L \subset H_i} a_i  < \frac{1}{d+1}  \sum_{i=1}^n a_i \,.
\end{equation}    
\end{definition}

\begin{remark}\label{rmk:essirr}
(i)
\((\mH, a)\) is stable if and only if the log pair 
\[
\left(\mathbb{CP}^{d} \ , \ \sum a_i' \, H_i \right) \hspace{2mm} \textup{ with } \hspace{2mm} 
a_i' = \frac{d+1}{\sum_{j=1}^n a_j} \, a_i 
\]
is klt, see \cite[\S 6.2]{borbonpanov}.
(ii) It is easy to show (see \cite[Lemma 2.27]{borbonpanov}) that if \((\mH, a)\) is stable then the arrangement \(\mH\) is essential and irreducible. 
\end{remark}

\begin{definition}
     A Hermitian metric \(\inn\) on \(\C^{d+1}\) is \emph{balanced} for \((\mH, a)\) if
    \begin{equation}\label{eq:balanced}
	\sum_{i=1}^{n} a_i \, P_i = \lambda \, \Id 
	\end{equation} 
    for some \(\lambda > 0\),
	where \(P_i\) is the orthogonal projection onto \(H_i^{\perp}\).
\end{definition}

\begin{remark}
Taking the trace in \eqref{eq:balanced} gives
\begin{equation}\label{eq:lambda}
\lambda=\frac{1}{d+1}\sum_{i=1}^n a_i \,.    
\end{equation}
Thus \(\inn\) is balanced if and only if
\[
\sum_{i=1}^n a_i\left(P_i-\frac{1}{d+1}\Id\right)=0.
\]
From the above equation, it follows that the Hermitian metric \(\inn\) is balanced if and only if the weighted barycentre of the normal lines \(H_i^\perp\) is zero under the moment-map embedding \(\CP^d\hookrightarrow\mathfrak{su}(d+1)^*\)
determined by \(\inn\); see \cite[Example 5.6]{szekelyhidi}.
\end{remark}

The following result is a consequence of the Kempf--Ness theorem.

\begin{proposition}\label{prop:stable}
Suppose that \((\mH,a)\) is stable. Then there exists a Hermitian metric \(\inn\) on \(\C^{d+1}\), unique up to scale, which is balanced for \((\mH,a)\).
\end{proposition}

\begin{proof}
By duality, the hyperplane arrangement \(\mH\) defines a point \(x\in(\CP^d)^n\). 
Suppose first that the weights \(a_i\) are positive integers.
Then the stability condition above is precisely the GIT-stability condition for
the point \(x\), with respect to the diagonal action of \(SL(d+1,\C)\) on
\((\CP^d)^n\) linearized by
\(\bigotimes_{i=1}^n \mathrm{pr}_i^*\mathcal O(a_i)\), see Theorem~11.2
in \cite{dolgachev}.

Fix the standard Hermitian metric \(\inn_0\) on \(\C^{d+1}\), and consider the corresponding moment map for the \(SU(d+1)\)-action on \((\CP^d)^n\) with the above linearization. By the Kempf--Ness theorem \cite[Theorem~5.19]{szekelyhidi}, there exists \(F\in SL(d+1,\C)\) such that \(F\cdot x\) is a zero of this
moment map. 
Since this moment map is the weighted sum of the standard moment maps on the factors, this means that the transformed configuration of normal lines has zero weighted barycentre with respect to \(\inn_0\). Equivalently, transporting the metric \(\inn_0\) back to the original configuration of normal lines gives a Hermitian metric on \(\C^{d+1}\) which is balanced for \((\mH,a)\).

For rational weights the result follows from the integer case just proved.
For arbitrary real weights, the proposition follows from the more general
results proved in \cite{kapovich}, see in particular Example~6.3 of that paper.
\end{proof}

We will also need the following converse result.

\begin{lemma}\label{lem:stable}
    Suppose that there exists a balanced metric \(\inn\) for \((\mH,a)\). Moreover, assume that \(\mH\) is irreducible. Then \((\mH,a)\) is stable.
\end{lemma}

\begin{proof}
    Let \(L \in \mL \setminus \{\emptyset, \CP^d\}\). We regard \(L\) as a vector  subspace of \(\C^{d+1}\) and denote by \(P_{L^{\perp}}\) the orthogonal projection onto \(L^{\perp}\) with respect to \(\inn\). Multiplying Equation \eqref{eq:balanced} on the right by \(P_{L^{\perp}}\), we obtain
    \[
    \sum_{i=1}^n a_i \, (P_i P_{L^{\perp}}) = \lambda \, P_{L^{\perp}} \,.
    \]
    Taking the trace,
    \(\sum_{i=1}^n a_i \, \tr (P_i P_{L^{\perp}}) = \lambda \, r(L)\).
    If \(L \subset H_i\) then \(P_i= P_iP_{L^{\perp}}\), so \(\tr (P_i P_{L^{\perp}}) =1\).
    For any \(i\) we have \(\tr (P_i P_{L^{\perp}}) \geq 0\) with equality if and only if \(L^{\perp} \subset H_i\). Thus,
    \[
    \sum_{i \,\mid\, L \subset H_i} a_i \leq \lambda \, r(L)
    \]
    with equality if and only if every \(H_i \in \mH\) contains either \(L\) or \(L^{\perp}\). Since \(\mH\) is irreducible, the strict inequality must hold. This, together with Equation \eqref{eq:lambda}, implies the strict stability inequality \eqref{eq:stable} as wanted.
\end{proof}

\subsection{The Hirzebruch quadratic form}

For an irreducible codimension-two subspace \(L \in \mG^2\) (see Section \ref{sec:ha}) we denote by \(a_L\) the linear function on \(\R^n\) given by
\begin{equation}\label{eq:aL}
   a_L \, (a_1, \ldots, a_n) = \frac{1}{2} \sum_{i \,|\, L \subset H_i} a_i  \,.
\end{equation}

\begin{definition}[{\cite[Definition 6.1]{borbonpanov}}]\label{def:hqf}
The \emph{Hirzebruch quadratic form} of \(\mH\) is the homogeneous degree-\(2\) polynomial on \(\R^n\) given by
\begin{equation}\label{eq:hir}
Q(a_1, \ldots, a_n) = \sum_{L \in \mG^2} a_L^2 \,-\, \frac{1}{2}  \sum_{i=1}^n B_i \, a_i^2 \,-\, \frac{1}{2(d+1)}  \left(\sum_{i=1}^{n} a_i\right)^2    \,,
\end{equation}
where \(B_i + 1\) is the number of elements \(L \in \mG^2\) with \(L \subset H_i\). 
\end{definition}

\subsection{Tight frames and the frame-potential inequality}\label{sec:tf}

Let \(\inn\) be a Hermitian metric on \(\C^r\) with \(r \geq 1\).

\begin{definition}\label{def:tightframe}
	A finite sequence of vectors \((v_i)_{i=1}^k\) in \(\C^r\)
	is a \emph{tight frame} for \(\C^r\) if for all \(x \in \C^r\),
	\begin{equation}\label{eq:tightframe}
	\sum_{i=1}^{k} \langle x, v_i \rangle v_i = \lambda \, x
	\end{equation}
    for some \(\lambda>0\).
\end{definition}

We recall the frame-potential inequality. For unit vectors, it is a particular instance of the \emph{Welch bounds}.  

\begin{proposition}[{\cite[Theorem 6.1]{waldron}}]\label{prop:tfineq}
	Let \((v_i)_{i=1}^k\) be a finite sequence of vectors in \(\C^r\) with at least one vector non-zero. Then
	\begin{equation}\label{eq:quaedineq}
	\sum_{i=1}^{k} \sum_{j=1}^k \big| \langle v_i, v_j \rangle \big|^2 \geq \frac{1}{r}  \left( \sum_{i=1}^{k} |v_i|^2 \right)^2 \,.
	\end{equation}
	Equality holds if and only if \((v_i)_{i=1}^k\) is a tight frame for \(\C^r\).
\end{proposition}

Finally, we note the straightforward relation between balanced Hermitian metrics and tight frames.

\begin{lemma}\label{lem:tfbalanced}
    Let \(\inn\) be a Hermitian metric on \(\C^{d+1}\).
    For each hyperplane \(H_i \in \mH\), choose a vector \(v_i \in \C^{d+1}\) such that \(v_i \in H_i^{\perp}\) and \(|v_i|^2=a_i\). Then the following are equivalent:
    \begin{enumerate}[label=\textup{(\roman*)}]
        \item \(\inn\) is balanced for \((\mH, a)\);
        \item \((v_i)_{i=1}^n\) is a tight frame for \(\C^{d+1}\).
    \end{enumerate}
\end{lemma}

\begin{proof}
    By definition of the vectors \(v_i\), for any \(x \in \C^{d+1}\), we have
    \[
    \langle x, v_i \rangle \, v_i = a_i P_i(x) \,,
    \]
    where \(P_i\) denotes the orthogonal projection onto \(H_i^{\perp}\) with respect to \(\inn\). Thus Equations \eqref{eq:balanced} and \eqref{eq:tightframe} are equivalent, so (i) \(\iff\) (ii).
\end{proof}

\subsection{Flat logarithmic connections and rank-two balanced metrics}

The following lemma is well-known; for a proof see
\cite[Proposition 3.8]{borbonpanov-pkc}.

\begin{lemma}\label{lem:flat}
    Let \(E\) be a complex vector space and let \(\nabla\) be the logarithmic connection on the trivial vector bundle over \(\C^{d+1}\) with fibre \(E\) given by
    \begin{equation}\label{eq:connection2}
       \nabla = d - \sum_{i=1}^n A_i \frac{d\ell_i}{\ell_i} 
    \end{equation}
    with \(A_i \in \End(E)\) and \(\ell_i\) defining linear equations of the hyperplanes \(H_i \in \mH\). Then the following are equivalent:
    \begin{enumerate}[label=\textup{(\roman*)}]
        \item the connection \(\nabla\) is flat;
        \item for every \(L \in \mL^2\), and every \(i\) such that \(L\subset H_i\), we have
        \begin{equation}\label{eq:kohno}
            \left[A_i\,,\sum_{j\,:\,L\subset H_j}A_j\right]=0 \,.
        \end{equation}
    \end{enumerate}
\end{lemma}

\begin{remark}
The commutation relations given by Equation~\eqref{eq:kohno} are known as
the \emph{Kohno relations}, after Kohno \cite{Kohno89}.    
\end{remark}

We now specialize to the case where the residues of the logarithmic connection are weighted orthogonal projections onto the normal lines to the hyperplanes.
Thus, after choosing a Hermitian metric \(\inn\) on \(\C^{d+1}\), we set \(A_i=a_i \,P_i\), where \(P_i\) is the orthogonal projection onto \(H_i^\perp\). In this case Equation~\eqref{eq:connection2}, with \(E=\C^{d+1}\),
defines a logarithmic connection on \(T\C^{d+1}\), written in the standard trivialization defined by the coordinate vector fields.

\begin{remark}
For residues of this form, the logarithmic connection \(\nabla\) is always torsion-free, since \(H_i=\ker A_i\); see \cite[Lemma 3.14]{borbonpanov-pkc}.
\end{remark}

We next characterize the Hermitian metrics for which the weighted orthogonal projections \(A_i = a_i \,P_i\) satisfy the Kohno relations.
Let \(L \in \mL\). Once a Hermitian metric is fixed, we identify the quotient \(\C^{d+1}/L\) with \(L^\perp\) and define
\[
\mH_L/L :=\{H_i\cap L^\perp \mid L\subset H_i\},
\]
thought of as an arrangement of linear hyperplanes in \(L^\perp\). We denote by \(a|_L\) the restriction of the weight vector \(a\) to the subarrangement \(\mH_L\subset\mH\).

\begin{definition}\label{def:ranktwo}
Let \(\inn\) be a Hermitian metric on \(\C^{d+1}\). We say that \(\inn\) is
\emph{rank-two balanced} for \((\mH,a)\) if the following conditions hold for
every \(L\in\mL^2\):
\begin{enumerate}[label=\textup{(\arabic*)}]
\item if \(L\in\mG^2\), then the induced Hermitian metric on \(L^\perp\) is
balanced for the weighted arrangement \((\mH_L/L,a|_L)\);
\item if \(L\in\mR^2\), say \(\mH_L=\{H_i,H_j\}\), then the normal lines
\(H_i^\perp\) and \(H_j^\perp\) are orthogonal.
\end{enumerate}
\end{definition}

\begin{remark}
In the real Euclidean setting, closely related rank-two balancing conditions
appear in Veselov's theory of \(\vee\)-systems; see
\cite{FeiginVeselov}. 
\end{remark}

\begin{lemma}\label{lem:ranktwokohno}
Let \(\inn\) be a Hermitian metric on \(\C^{d+1}\), and let \(A_i=a_iP_i\), where \(P_i\) is the orthogonal projection onto \(H_i^\perp\). Then \(\inn\) is rank-two balanced for \((\mH,a)\) if and only if the
endomorphisms \(A_i\) satisfy the Kohno relations \eqref{eq:kohno}.
\end{lemma}

\begin{proof}
Let \(L\in\mL^2\) and set
\[
S_L=\sum_{i\,:\,L\subset H_i} A_i \,.
\]
The Kohno relations at \(L\) say that \([A_i,S_L]=0\) for every \(i\) such that \(L\subset H_i\).

Suppose first that \(L\in\mG^2\). Then \(m_L\geq 3\). The relations \([A_i,S_L]=0\) imply that the restriction of \(S_L\) to \(L^{\perp}\) preserves the normal lines \(H_i^\perp\subset L^\perp\), for all \(L\subset H_i\). Since an endomorphism of a
two-dimensional complex vector space preserving three distinct lines is a scalar multiple of the identity, we obtain \(S_L=\lambda P_{L^\perp}\) for some \(\lambda>0\). 
Conversely, if \(S_L=\lambda P_{L^\perp}\), then \(S_L\) commutes with each \(A_i\) such that \(L\subset H_i\), and hence the Kohno relations at \(L\) hold. Thus the Kohno relations at \(L\) are equivalent to the balanced condition for \((\mH_L/L,a|_L)\).

Suppose now that \(L\in\mR^2\), say \(\mH_L=\{H_i,H_j\}\). Then the Kohno
relations at \(L\) reduce to \([A_i,A_j]=0\), which is equivalent to the
orthogonality of the two normal lines \(H_i^\perp\) and \(H_j^\perp\).
\end{proof}

\begin{lemma}\label{lem:tfkohno}
    Let \(\inn\) be a Hermitian metric on \(\C^{d+1}\). For each hyperplane
    \(H_i \in \mH\), choose a vector \(v_i \in H_i^{\perp}\) such that
    \(|v_i|^2=a_i\). Then the following conditions are equivalent:
    \begin{enumerate}[label=\textup{(\roman*)}]
        \item \(\inn\) is rank-two balanced for \((\mH,a)\);
        \item for every \(L \in \mL^2\) we have:
        \begin{enumerate}[label=\textup{(\arabic*)}]
            \item if \(L \in \mG^2\), then the vectors \(\{v_i \mid L \subset H_i\}\)
            form a tight frame for \(L^{\perp}\);
            \item if \(L \in \mR^2\), say \(\mH_L=\{H_i,H_j\}\), then \(\langle v_i,v_j\rangle=0\).
        \end{enumerate}
    \end{enumerate}
\end{lemma}

\begin{proof}
This follows directly from 
Lemmas \ref{lem:tfbalanced} and \ref{lem:ranktwokohno}.
\end{proof}

Finally, we show that, under the mild assumption that the arrangement is
essential and irreducible, rank-two balanced implies balanced.

\begin{lemma}\label{lem:kohnotf}
    Assume that \(\mH\) is essential and irreducible and suppose that \(\inn\)
    is rank-two balanced for \((\mH,a)\). Then \(\inn\) is balanced for
    \((\mH,a)\).
\end{lemma}

\begin{proof}
By Lemma~\ref{lem:ranktwokohno}, the endomorphisms \(A_i=a_iP_i\) satisfy the Kohno relations. Let \(S=\sum_{i=1}^n a_iP_i\). It is easy to show (see \cite[Lemma 3.9]{borbonpanov-pkc}) that the Kohno
relations imply \([S,P_i]=0\) for all \(i\). In particular, the linear
endomorphism \(S\) preserves all the hyperplanes \(H_i\in\mH\). Since \(\mH\)
is essential and irreducible, by \cite[Lemma 2.33]{borbonpanov-pkc} the
endomorphism \(S\) is a scalar multiple of the identity. Since \(S\) has positive
trace, this scalar is positive, and hence \(\inn\) is balanced for \((\mH,a)\).
\end{proof}

\section{Proof of Theorem~A}\label{sec:proof}

Throughout this section, we assume that \((\mH, a)\) is stable and we let \(Q\) be the Hirzebruch quadratic form of \(\mH\).

\subsection{The Hirzebruch quadratic form is non-positive}

\begin{lemma}\label{lem:hirineq}
    The inequality \eqref{eq:hirineq} holds.
\end{lemma}

\begin{proof}
    Since  \((\mH, a)\) is stable, by Proposition \ref{prop:stable} we can choose a balanced metric for \((\mH,a)\). For each \(1 \leq i \leq n\) choose a vector \(v_i \in H_i^{\perp}\) with \(|v_i|^2 = a_i\). By Lemma \ref{lem:tfbalanced} the sequence of vectors \((v_i)_{i=1}^n\) is a tight frame for \(\C^{d+1}\). By Proposition \ref{prop:tfineq} (with \(r=d+1\)) we have
	
\begin{equation}\label{eq:hirpf1}
	\frac{1}{d+1}  \left( \sum_{i=1}^{n} a_i \right)^2 =
	\sum_{i=1}^{n} \sum_{j=1}^{n} \big| \langle v_i, v_j \rangle\big|^2 \,.	
\end{equation}
	
The assumption that the hyperplanes in \(\mH\) are pairwise distinct implies that if \(i \neq j\) and \(L=H_i \cap H_j\) then either \(L \in \mR^2\) or  \(L \in \mG^2\) (see Section \ref{sec:ha} for notation). Let \(R\) be the subset of pairs \((i,j)\) such that \(H_i \cap H_j \in \mR^2\). 
For \(L \in \mG^2\) let \(I(L)\) be the subset of all indices \(i \in \{1, \ldots, n\}\) such that \(L \subset H_i\).
Then,
	\begin{equation}\label{eq:hirpf2}
	\begin{gathered}
	\sum_{i=1}^{n} \sum_{j=1}^{n}
	\big| \langle v_i, v_j \rangle\big|^2 = \sum_{(i,j) \in R}  \big|\langle v_i, v_j \rangle\big|^2 \\
	\,+ \, \sum_{L \in \mG^2} \sum_{\,\,i,j \in I(L)} \big|\langle v_i, v_j \rangle\big|^2 \,-\, \sum_{i=1}^{n} B_i \, a_i^2 \,,
	\end{gathered}	
	\end{equation}
where \(B_i + 1\) is the number of elements \(L \in \mG^2\) with \(L \subset H_i\). 

For \(L \in \mG^2\), Proposition \ref{prop:tfineq} (with \(r=2\)) applied to the sequence of vectors \((v_i)_{i \in I(L)}\) in \(L^{\perp}\) gives
	\begin{equation}\label{eq:hirpf3}
	\sum_{i, j \in I(L)} \big|\langle v_i, v_j \rangle\big|^2 \geq \frac{1}{2} \left(\sum_{i \in I(L)} a_i\right)^2 = 2 \, a_L^2 \,.
	\end{equation}

Equations \eqref{eq:hirpf1}, \eqref{eq:hirpf2} and \eqref{eq:hirpf3} imply 
	\[
	\frac{1}{d+1}  \left( \sum_{i=1}^{n} a_i \right)^2 \geq 2 \, \sum_{L \in \mG^2} a_{L}^2 \,-\, \sum_{i=1}^{n} B_i \, a_i^2 \,,
	\]
which rearranges to inequality \eqref{eq:hirineq}.
\end{proof}

\subsection{The equality case}

\begin{lemma}\label{lem:qa0}
    Let \(\inn\) be a balanced metric for \((\mH, a)\). Then
    \(Q(a)=0\) if and only if
    \(\inn\) is rank-two balanced for \((\mH, a)\).
\end{lemma}

\begin{proof}
The proof of Lemma~\ref{lem:hirineq} shows that \(Q(a)=0\) if and only if
equality holds in all the frame-potential inequalities \eqref{eq:hirpf3} for
\(L\in\mG^2\), and all the terms \(|\langle v_i,v_j\rangle|^2\) with
\(H_i\cap H_j\in\mR^2\) vanish. By Proposition \ref{prop:tfineq} and Lemma~\ref{lem:tfkohno}, this is
equivalent to \(\inn\) being rank-two balanced for \((\mH,a)\).
\end{proof}

\begin{lemma}\label{lem:qa02}
    Suppose that there exists a rank-two balanced metric \(\inn'\) for \((\mH,a)\). Then \(Q(a)=0\).
\end{lemma}

\begin{proof}
    The arrangement \(\mH\) is essential and irreducible, see Remark~\ref{rmk:essirr} (ii).
    Therefore, by Lemma \ref{lem:kohnotf} 
    the metric \(\inn'\) is balanced for \((\mH, a)\). By Lemma~\ref{lem:qa0}, \(Q(a)=0\).
\end{proof}

We can now prove the main result of the paper.

\begin{proof}[Proof of Theorem~A]
    The inequality \eqref{eq:hirineq} follows from Lemma \ref{lem:hirineq}. For the equality case,
    by Lemmas \ref{lem:flat} and \ref{lem:ranktwokohno}, it is enough to show that \(Q(a)=0\) if and only if there is a rank-two balanced metric for \((\mH,a)\).
    In one direction,
    if \(Q(a)=0\) then by Lemma \ref{lem:qa0} the balanced metric \(\inn\) is rank-two balanced for \((\mH, a)\). In the other direction, if there exists a rank-two balanced metric for \((\mH, a)\), then by Lemma~\ref{lem:qa02} \(Q(a)=0\). 
    The uniqueness up to scale of the rank-two balanced metric follows from Lemma \ref{lem:kohnotf} together with the uniqueness part of Proposition \ref{prop:stable}.
\end{proof}

\section{Applications}\label{sec:applications}

We record three consequences of Theorem~A, illustrating how the result can be translated into concrete arrangement-theoretic information.

\subsection{Langer's bound on multiplicities}

We consider the symmetric case where all weights are equal. 
Stability becomes a purely combinatorial condition on the multiplicities of the intersections of \(\mathcal H\), and Theorem~A then gives a lower bound for the total multiplicity of the codimension-two strata. 

\begin{corollary}\label{cor:langer}
    Let \(\mH = \{H_1, \ldots, H_n\}\) be a hyperplane arrangement in \(\CP^d\) consisting of \(n > d+1\) hyperplanes. Suppose that for every element \(L\) of the intersection poset \(\mL\) of \(\mH\) with \(L \neq \emptyset, \CP^d\) we have
    \begin{equation}\label{eq:mL}
    m_L < \frac{r(L)}{d+1} \, n \,.    
    \end{equation}
    Then the following inequality holds:
    \begin{equation}\label{eq:mulbound}
        \sum_{L \in \mL^2} m_L \geq \left(1 - \frac{2}{d+1} \right)n^2 + n \,.
    \end{equation}
    Equality holds if and only if every hyperplane \(H_i \in \mH\) intersects \(\mH \setminus \{H_i\}\) along \((1-2/(d+1)) n + 1\) codimension-two subspaces. Furthermore, equality holds if and only if there exists a rank-two balanced metric for \((\mH, \mathbf{1})\), where \(\mathbf{1}\) is the vector in \(\R^n\) with all components equal to \(1\).
\end{corollary}

\begin{proof}
    The assumption \eqref{eq:mL} implies that \((\mH,\mathbf{1})\) is stable. The inequality \eqref{eq:mulbound} and the first equality characterization follow as in the proof of \cite[Theorem 7.14]{borbonpanov}. The second equality characterization follows from Theorem~A.
\end{proof}

\begin{remark}
    When \(d=2\), the bound \eqref{eq:mulbound} was derived in \cite[Proposition 11.3.1]{langer} from a more general Miyaoka--Yau inequality for log canonical surface pairs and in \cite[Corollary 7.9]{panov} from the Bogomolov--Gieseker inequality for stable parabolic bundles. The proof presented here is more elementary.
\end{remark}

\subsection{A matroid-polytope criterion}\label{sec:linearconst}

The matroid polytope, or base polytope, of an essential and irreducible arrangement \(\mathcal H = \{H_1, \ldots, H_n\}\) is the convex hull \(B\subset\R^n\) of the indicator functions of subsets \(I\subset\{1,\ldots,n\}\) such that \(|I|=d+1\) and the hyperplanes \(\{H_i\mid i\in I\}\) are linearly independent.\footnote{
Matroid polytopes play a central role in the theory of weighted stable hyperplane arrangements; see, for example, Alexeev's book \cite{alexeev2015}.
}
Under our assumptions, \(\dim B=n-1\); see \cite[Theorem 1.12.9]{borovikgelfandwhite}.

The next corollary turns the existence problem for rank-two balanced metrics into an explicit intersection problem between the relative interior of the matroid polytope and the kernel of the Hirzebruch quadratic form.

\begin{corollary}\label{cor:matroid}
    Let \(\mathcal H\) be essential and irreducible, and let \(a\in\R^n_{>0}\). Then the following conditions are equivalent.
	\begin{enumerate}[label=\textup{(\arabic*)}]
        \item There exists a rank-two balanced metric for \((\mH, a)\).
        \item \((\mH, a)\) is stable and \(Q(a)=0\).
        \item If \(s= \sum_{i=1}^n a_i\) then
        \[
        \frac{d+1}{s} \, a \in B^{\circ} \cap \ker Q \,,
        \]
        where \(\ker Q\) denotes the kernel of the symmetric bilinear form associated with \(Q\) and \(B^{\circ}\) is the relative interior of \(B\) inside the affine hyperplane \(\{s=d+1\} \subset \R^n\).
	\end{enumerate}
\end{corollary}

\begin{proof}
    (1) \(\implies\) (2). The fact that \((\mH, a)\) is stable follows from Lemmas \ref{lem:kohnotf} and \ref{lem:stable}.
    Since \((\mH, a)\) is stable and \(\inn\) is a rank-two balanced metric for \((\mH, a)\), by Theorem~A, \(Q(a) = 0\). \vspace{.5em}

    (2) \(\implies\) (1). This is part of Theorem~A. \vspace{.5em}

    (2) \(\iff\) (3). Let \(\mathcal C=\R_{>0}\cdot B^\circ\). This is precisely the open cone of weights \(a\) for which \((\mathcal H,a)\) is stable; see \cite[\S 6.3]{borbonpanov}. By Theorem~A, \(Q\leq 0\) on \(\mathcal C\). Since \(\mathcal C\) is open, any point \(a\in\mathcal C\) with \(Q(a)=0\) is a local maximum of \(Q\), hence a critical point. Since \(Q\) is quadratic, this is equivalent to \(a\in\ker Q\). Finally, if \(s=\sum_i a_i\), then \(a\in\mathcal C\) is equivalent to \(\frac{d+1}{s}a\in B^\circ\).
    This proves the equivalence of (2) and (3).
\end{proof}

\subsection{Localization in the vanishing case}

Polyhedral K\"ahler metrics are our main motivation for studying the Hirzebruch quadratic form. In \cite{borbonpanov-pk}, the condition \(Q(a)=0\) appears as the numerical equality condition for the existence problem on \(\CP^d\). 
The following corollary is the form in which Theorem~A is used there: it shows that the equality condition is inherited by the irreducible localizations of the arrangement.

\begin{corollary}\label{cor:loc}
    Suppose that \((\mH, a)\) is stable and that \(Q(a)=0\). Let \(L \in \mL\) be an irreducible subspace. Let \(\mH_L /L = \{H_i/L \,\mid\, L \subset H_i\}\) be the essentialization of the localized arrangement and let \(a|_L \in \R^{\mH_L}\) be the vector with components \(a_i\) if \(L \subset H_i\). Then the following holds:
    \begin{enumerate}[label=\textup{(\roman*)}]
        \item the weighted arrangement \((\mH_L/L, a|_L)\) is stable;
        \item \(Q_L (a|_L) = 0\), where \(Q_L\) is the Hirzebruch quadratic form of \(\mH_L/L\).
    \end{enumerate}
\end{corollary}

\begin{proof}
    By Theorem~A, there is a rank-two balanced metric \(\inn\) for \((\mH,a)\). It is straightforward to check that the restriction of \(\inn\) to \(L^\perp\) is a rank-two balanced metric for \((\mH_L/L,a|_L)\). The result then follows from the direction (1) \(\implies\) (2) of Corollary \ref{cor:matroid}.
\end{proof}

\bibliographystyle{alpha}
\bibliography{refs}

%\vspace{-3mm}

\end{document}